\documentclass{amsart}

\usepackage{amsmath,amssymb,verbatim}
\usepackage{amsthm}
\usepackage{enumerate}
\usepackage{url}
\usepackage{mathtools}
\usepackage{tikz-cd}

\newtheorem{theorem}{Theorem}[section]

\newtheorem{corollary}[theorem]{Corollary}

\newtheorem{fact}[theorem]{Fact}
\newtheorem{claim}[theorem]{Claim}

\theoremstyle{definition}
\newtheorem{definition}[theorem]{Definition}

\def\seq{\subseteq}

\def\inv{^{\text{-}1}}

\def\cU{\mathcal{U}}

\def\cC{\mathcal{C}}

\def\N{\mathbb{N}}

\def\R{\mathbb{R}}

\def\Z{\mathbb{Z}}
\def\C{\mathbb{C}}

\def\BD{\operatorname{\mathsf{BD}}}

\def\bone{\boldsymbol{1}}
\def\LM{\mathfrak{M}_\ell}

\newcommand{\sclqed}{\hfill$\dashv_{\text{\scriptsize{subclaim}}}$}

\AtBeginDocument{%
   \def\MR#1{}
}

\title{Stable functions and F{\o}lner's Theorem}

\date{June 17, 2025}

\author[G. Conant]{Gabriel Conant}

\thanks{Partially supported by   NSF grant DMS-2452816}


\address{Department of Mathematics, Statistics, and Computer Science\\
University of Illinois Chicago
}
\email{gconant@uic.edu}

\begin{document}

\begin{abstract}
We show that if $G$ is an amenable group and $A\seq G$ has positive upper Banach density, then there is an identity neighborhood $B$ in the Bohr topology on $G$ that is almost contained in $AA\inv$ in the sense that $B\backslash AA\inv$ has upper Banach density $0$. This generalizes the abelian case (due to F{\o}lner) and the countable case (due to Beiglb\"{o}ck, Bergelson, and Fish). The proof is indirectly   based on local stable group theory in continuous logic. The main ingredients are Grothendieck's double-limit characterization of relatively weakly compact sets in spaces of continuous functions, along with results of Ellis and Nerurkar on the topological dynamics of weakly almost periodic flows.
\end{abstract}

\maketitle

\section{Introduction}

Let $G$ be a (discrete) group. Following \cite{BBF}, a subset $B$ of $G$ is called a \emph{Bohr set} if it contains a set of the form $\tau\inv(U)$ where $\tau\colon G\to K$ is a homomorphism to  a compact Hausdorff group with dense image, and $U\seq K$ is open and nonempty. If, moreover, $U$ contains the identity of $K$, then $B$ is called a \emph{Bohr}$_0$ set.

Recall that $G$ is \emph{amenable} if there is a left-invariant finitely additive probability measure on the Boolean algebra of subsets of $G$. We write $\LM(G)$ for the collection of such measures (so $G$ is amenable if and only if $\LM(G)\neq \emptyset$).  If $G$ is amenable then any set $A\seq G$ has an \emph{upper Banach density}, which can be defined as
\[
\BD(A)=\sup\{\mu(A):\mu\in \LM(G)\}.
\]
(This  differs from the standard definition of Banach density in terms of \emph{F{\o}lner sequences/nets}.  See  Section \ref{sub:BDsup} for details and discussion.)

A 1954 result of F{\o}lner \cite{FolB1,FolB2} shows that if $G$ is abelian (hence amenable) and $A\seq G$ has positive upper Banach density, then there is a Bohr$_0$ set $B$ in $G$ that is \emph{almost} contained in $A-A$ in the sense that $\BD(B\backslash (A-A))=0$. This bears some analogy to Steinhaus's Theorem \cite{Steinhaus}, which says that if $A\seq \R$ has positive Lebesgue measure, then $A-A$ contains an open interval around $0$. See Section \ref{sub:Steinhaus} for further discussion on this comparison. 

In 2010, Beiglb\"{o}ck, Bergelson, and Fish \cite{BBF} proved the analogue of F{\o}lner's result for any countable amenable group (but with a Bohr set rather than a Bohr$_0$ set; see \cite[Corollary 5.3]{BBF}). This result was actually a step along the way to their generalization of Jin's Theorem \cite{JinAB}  to countable amenable groups (see Section \ref{sub:Jin}). The methods of \cite{BBF} leverage countability of $G$ in order to use certain tools from ergodic theory. In this short paper, we give a  proof of F{\o}lner's Theorem for arbitrary amenable groups  using (relatively) classical results from functional analysis together with 1980's  topological dynamics. In particular, we prove the following theorem.

\begin{theorem}\label{thm:main}
Let $G$ be an amenable group and suppose $A\seq G$ has positive upper Banach density. Then there is a Bohr$_0$ set $B\seq G$ such that $\BD(B\backslash AA\inv)=0$.
\end{theorem}

The main tools for our proof will be Grothendieck's \cite{GroWAP} characterization of relatively weakly compact sets in Banach spaces of  continuous functions (Theorem \ref{thm:Gro}), and structural results of Ellis and Nerurkar \cite{EllNer} on the topological dynamics of weakly almost periodic flows (Theorem \ref{thm:EN}). This draws a fundamental connection to the work of F{\o}lner \cite{FolB1,FolB2},  and of Bieglb\"{o}ck, Bergelson, and Fish \cite{BBF}, where almost periodic functions  play a central role.

 Other than the  tools from \cite{GroWAP} and \cite{EllNer} mentioned above, the proof of Theorem \ref{thm:main} will be almost entirely based on elementary topological arguments.
That being said, the overall strategy of the proof is heavily inspired by recent work of the author and Pillay \cite{CP-AVSAR} on the structure of ``stable functions" on groups. For this reason, the argument has underpinnings in continuous model theory, even though this will not appear explicitly in the proof. In Section \ref{sub:MT}, we will provide further details on this connection for readers versed in model theory (in fact, these readers may benefit from reading this section before the proof of Theorem \ref{thm:main}).

From the model-theoretic perspective, Theorem \ref{thm:main} can be seen as an application of local stable group theory in continuous logic.  This makes it worthwhile to point out that Hrushovski's Stabilizer Theorem \cite{HruAG} directly implies a weaker version of Theorem \ref{thm:main} in which $B\backslash AA\inv$ only has \emph{lower} Banach density $0$ (which is equivalent to saying $AA\inv$ is ``piecewise Bohr$_0$" in the sense of \cite{BBF}). For details, see the proof of \cite[Theorem 5.6]{CHP}. The improvement here to upper Banach density $0$ will result from ``unique ergodicity" of weakly almost periodic flows (proved by Ellis and Nerurkar \cite{EllNer}), which a model theorist will recognize as the fact that a left-invariant stable formula  in a first-order expansion of a group (in continuous logic) admits a unique left-invariant Keisler measure/functional. It could be interesting to investigate what further bearing  unique ergodicity has on the Stabilizer Theorem itself. 

The paper is organized as follows. In Section \ref{sec:prelim} we review   preliminaries from functional analysis and topological dynamics needed for Theorem \ref{thm:main}. The proof of Theorem \ref{thm:main} then follows in Section \ref{sec:proof}. Section \ref{sec:remarks} elaborates at some length on several points of discussion initiated above.

\section{Preliminaries}\label{sec:prelim}

Given a topological space $S$, we let $\cC_b(S)$  denote the space of bounded continuous complex-valued  functions on $S$. Recall that $\cC_b(S)$ is a Banach space with the sup norm, and in fact a unital $C^*$-algebra with the involution operation from complex conjugation. 
If $S$ is compact then we omit the subscript $b$ and write $\cC(S)$. 
The \emph{weak topology} on $\cC_b(S)$ is the weakest topology for which all bounded linear functionals on $\cC_b(S)$ are continuous.  The following is \cite[Th\'{e}or\`{e}me 6]{GroWAP}.

\begin{theorem}[Grothendieck]\label{thm:Gro}
Let $S$ be a topological space with a fixed dense subset $S_0\seq S$. Then a set $X\seq \cC_b(S)$ is relatively weakly compact if and only if $X$ is uniformly bounded and, for any sequences $(f_i)_{i=0}^\infty$ from $X$ and $(x_j)_{j=0}^\infty$ from $S_0$,
\[
\textstyle\lim_i\lim_j f_i(x_j)=\lim_j\lim_i f_i(x_j)
\]
whenever both limits exist.
\end{theorem}

Using Banach-Alaoglu and basic facts about Hilbert spaces, one obtains the following well-known consequence of the previous theorem.

\begin{corollary}\label{cor:Gro}
Let $H$ be a Hilbert space. Then for any sequences $(x_i)_{i=0}^\infty$ and $(y_j)_{i=0}^\infty$ from the unit ball of $H$, 
\[
\textstyle\lim_i\lim_j\langle x_i,y_j\rangle=\lim_j\lim_i\langle x_i,y_j\rangle
\]
whenever both limits exist.
\end{corollary}

Now let $G$ be a (discrete) group. A \textit{$G$-flow} is a compact Hausdorff space $S$ together with a (left) action of $G$ by homeomorphisms.  In this case, a \textit{subflow} of $S$ is a nonempty closed $G$-invariant subset of $S$. A subflow of $S$ is \textit{minimal} if it does not properly contain another subflow of $S$. 

\begin{definition}\label{def:E}
Let $S$ be a $G$-flow. Given $g\in G$, we use $\breve{g}$ to denote the action map by $g$ on $S$. The \textbf{Ellis semigroup} of $S$, denoted $E(S)$, is the closure of the set $\{\breve{g}:g\in G\}$ in the space of functions $S^S$ (with the product topology).
\end{definition}

 Given a $G$-flow $S$, it is straightforward to check that $E(S)$ is a $G$-flow under the action $gp\coloneqq \breve{g}\circ p$, and  a semigroup under composition of functions. Moreover, under the induced topology, $E(S)$ is a \emph{right topological semigroup} (i.e., multiplication on the right by any fixed point in $E(S)$ is continuous).  For more on the basic theory of Ellis semigroups, see \cite{EllLTD}.

Note that if $S$ is a $G$-flow, then we have a natural action of $G$  on $\cC(S)$  via $g \varphi(x)=\varphi(g\inv x)$. 

\begin{definition}
Let $S$ be a $G$-flow. Then a function in $\cC(S)$ is \textbf{weakly almost periodic} if its $G$-orbit is relatively weakly compact in $\cC(S)$. The $G$-flow $S$ is called \textbf{weakly almost periodic} if every function in $\cC(S)$ is weakly almost periodic.
\end{definition}

We now summarize several  results from \cite{EllNer} on weakly almost periodic flows. It is worth noting that Theorem \ref{thm:Gro} is a key ingredient in the proofs of these results. 

\begin{theorem}[Ellis \& Nerurkar]\label{thm:EN}
Suppose $S$ is a weakly almost periodic $G$-flow.
\begin{enumerate}[$(a)$]
\item $E(S)$ has a unique minimal subflow $K$. Moreover, $K$ is the unique left ideal in $E(S)$ and a compact group (under the semigroup operation).
\item The identity in $K$ commutes with every element of $E(S)$.
\item $E(S)$ admits a unique $G$-equivariant regular Borel probability measure.
\end{enumerate}
\end{theorem}

Parts $(a)$ and $(b)$ are contained in \cite[Proposition 11.5]{EllNer} (though we caution the reader that Ellis and Nerurkar use right actions rather than left). Part $(c)$ follows from \cite[Proposition 11.10]{EllNer}, which actually says that if $S$ has a unique minimal subflow then $S$ admits a unique $G$-invariant regular Borel probability measure. So  we are really applying this statement to the $G$-flow $E(S)$, which is also weakly almost periodic (e.g., this follows from \cite[Proposition II.2]{EllNer} and the basic exercise that $E(S)$ and $E(E(S))$ are isomorphic).

Finally, if  $S$ is a $G$-flow then a subset $X\seq S$ is called \textbf{generic}  if $S=FX$ for some finite $F\seq G$. A point $x\in S$ is called \textbf{generic} if every open set containing $x$ is generic. The next fact is a straightforward topology exercise (see also Lemma 1.7 and Corollary 1.9 in \cite{NewTD}).

\begin{fact}\label{fact:generic}
Let $S$ be a $G$-flow. The following are equivalent.
\begin{enumerate}[$(i)$]
\item $S$ has a unique minimal subflow.
\item Every generic open subset of $S$ contains a generic point.
\item The set of generic points in $S$ is the unique minimal subflow of $S$.
\end{enumerate}
\end{fact}

\section{Proof of Theorem \ref{thm:main}}\label{sec:proof}

Let $G$ be an amenable group. We will make occasional use of the space $\beta G$ of ultrafilters on $G$, i.e., the Stone-\v{C}ech compactification of $G$ (as a discrete space). Given $X\seq G$, let $[X]$ denote the corresponding basic clopen set in $\beta G$ consisting of the ultrafilters that contain $X$. Note that $\beta G$ is a $G$-flow under the natural action. Recall also that any $\mu\in \LM(G)$  uniquely determines  a $G$-invariant regular Borel probability measure $\mu$ on $\beta G$ with the property that $\mu(X)=\mu([X])$ for any $X\seq G$.

Now fix a set $A\seq G$ with $\BD(A)>0$. We want to find a Bohr$_0$ set $B\seq G$ such that $\BD(B\backslash AA\inv)=0$. The proof proceeds in two main steps. \medskip

\noindent\textbf{Step 1.} In this step, we construct the group compactification $\tau\colon G\to K$ that will be used (in step 2) to obtain the desired Bohr$_0$ set $B$.

 Since $\BD(A)>0$, there is a measure $\mu\in\LM(G)$ such that $\mu(A)>0$. Define the function $\varphi\colon G\to [0,1]$ so that  $\varphi(x)=\mu(A\cap xA)$.
 
Note that $[0,1]^G$ is a $G$-flow under the action $g\psi(x)=\psi(g\inv x)$ where $g\in G$ and $\psi\colon G\to [0,1]$. Let $S$ be the $G$-flow given by the orbit closure of $\varphi$ in $[0,1]^G$.

\begin{claim}\label{cl:WAP}
$S$ is weakly almost periodic.
\end{claim}
\begin{proof}
Define $\tilde{\varphi}\colon S\to [0,1]$ so that $\tilde{\varphi}(\psi)=\psi(1)$.
Then $\tilde{\varphi}\in\cC(S)$ and the $G$-orbit  of $\tilde{\varphi}$ in $\cC(S)$ separates points in $S$. So by Stone-Weierstrass, $\cC(S)$ is the (uniform) closure of the unital $C^*$-algebra generated by the orbit of $\tilde{\varphi}$.  
Also, the collection of weakly almost periodic functions in $\cC(S)$ is a closed $G$-invariant unital $C^*$-subalgebra (see, e.g., \cite[Theorem 1.2]{Pourab}). Hence it suffices to show that $\tilde{\varphi}$ is weakly almost periodic. Toward a contradiction, suppose $\tilde{\varphi}$ is not weakly almost periodic. We apply Theorem \ref{thm:Gro} to $S$, taking   $X$ to be the $G$-orbit of $\tilde{\varphi}$, and  $S_0$ to be the $G$-orbit of $\varphi$.  This yields sequences $(a_i)_{i=0}^\infty$ and $(b_j)_{j=0}^\infty$ from $G$ such that $\lim_i\lim_j \varphi(a_i\inv b_j)\neq \lim_j\lim_i \varphi(a_i\inv b_j)$ (and both limits exist).

Let $H$ be the Hilbert space $L^2_{\C}(\beta G,\mu)$ with inner product $\langle x,y\rangle=\int x\overline{y}\,d\mu$. For $i,j\geq 0$, let $x_i=\boldsymbol{1}_{[a_i A]}$ and $y_j=\boldsymbol{1}_{[b_j A]}$. Then  $x_i,y_j$ are in the unit ball of $H$, and 
\[
\varphi(a_i\inv b_j)=\mu(a_iA\cap b_jA)=\langle x_i ,y_j\rangle.
\]
So $\lim_i\lim_j \langle x_i,y_j\rangle\neq\lim_j\lim_i \langle x_i,y_j\rangle$, which contradicts Corollary \ref{cor:Gro}.
\end{proof}

Let $E=E(S)$ be the Ellis semigroup of $S$.  By Claim \ref{cl:WAP} and Theorem \ref{thm:EN}$(a)$, $E$ has a unique minimal subflow $K$, which is also the unique left ideal in $E$ and a compact group. Let $u$ be the identity in $K$. Since $K$ is a left ideal, we can define a map $\tau_*\colon E\to K$ so that $\tau_*(p)=p\circ u$. \medskip

\begin{claim}\label{cl:Pihom}
$\tau_*$ is a continuous surjective semigroup homomorphism. Moreover, if $g\in G$ and $p\in E$ then $\tau_*(gp)=g\tau_*(p)$. 
\end{claim}
\begin{proof} 
Note that $\tau_*$ is continuous since $E$ is a right topological semigroup, and surjective since $\tau_*$ restricts to the identity on $K$.
Also, given $p,q\in E$, we have 
 \[
\tau_*(p\circ q)=p\circ q\circ u=p\circ q \circ u\circ u=p \circ u\circ q\circ u=\tau_*(p)\circ \tau_*(q),
 \]
 where the third equality uses Theorem \ref{thm:EN}$(b)$. Finally, given $g\in G$ and $p\in E$, we have $\tau_*(gp)=gp\circ u=\breve{g}\circ p\circ u=\breve{g}\circ \tau_*(p)=g\tau_*(p)$.
\end{proof}

Now define $\tau\colon G\to K$ so that $\tau(g)=gu$. Then for $g\in G$, we have $\tau(g)=\tau_*(\breve{g})$ (recall Definition \ref{def:E}). This allows us to view $\tau_*$ as an ``extension" of $\tau$ factoring through $g\mapsto \breve{g}$. In particular, it follows from Claim \ref{cl:Pihom} that $\tau$ is a group homomorphism. The image of $\tau$ is the $G$-orbit of $u$, which is dense in $K$ since $K$ is a minimal subflow of $E$. Altogether, $\tau\colon G\to K$ is a group compactification of $G$. \medskip

\noindent\textbf{Step 2.} In this step, we will use the compactification $\tau$ (from step 1) to obtain a Bohr$_0$ set $B$ in $G$ satisfying $\BD(B\backslash AA\inv)=0$. 

 Let $\cU$ be the collection of open identity neighborhoods of $u$ in $K$. Given $U\in\cU$, we write $\overline{U}$ for the  closure of $U$ in $K$. The next claim is a standard exercise, but we include the proof since it is not typically presented in this form.

\begin{claim}\label{cl:generic}
If $U\in\cU$ then $\tau_*\inv(U)$ is generic in $E$.
\end{claim}
\begin{proof}
Fix $U\in\cU$. We first show $K=GU$. Fix $x\in K$. Then $Ux\inv$ is a nonempty open set and hence contains $gu$ for some $g\in G$. So $gu=vx\inv$ for some $v\in U$. So $gux=v$, i.e., $gx=v$, i.e., $x=g\inv v\in GU$. 

Since $K$ is compact, there is some finite $F\seq G$ such that $K=FU$. Fix $p\in E$. Then $\tau_*(p)=gv$ for some $g\in F$ and $v\in U$. So $\tau_*(g\inv p)=g\inv \tau_*(p)=v\in U$ (recall Claim \ref{cl:Pihom}), hence $g\inv p\in \tau_*\inv(U)$, i.e., $p\in g \tau_*\inv(U)$. This shows $E=F\tau_*\inv(U)$.
\end{proof}

We now ``extend" $\varphi$ to $\hat{\varphi}\colon E\to [0,1]$ by setting $\hat{\varphi}(p)=p(\varphi)(1)$. Note that $\hat{\varphi}$ is continuous and if $g\in G$ then $\hat{\varphi}(\breve{g})=\varphi(g\inv)=\varphi(g)$.

\begin{claim}\label{cl:upos}
$\hat{\varphi}(u)>0$.
\end{claim}
\begin{proof}
More specifically, we show $\hat{\varphi}(u)\geq \alpha\coloneqq \frac{1}{2}\mu(A)^2$. (This bound can be improved; see Section \ref{sec:popdiff}.) Let $\epsilon<\alpha$ be arbitrary. Define $X=\{p\in E:\hat{\varphi}(p)>\epsilon\}$, and note that $X$ is open in $E$.  \medskip

\noindent\textit{Subclaim:} If $U\in\cU$ then $X\cap \tau_*\inv(U)$ is generic in $E$.

\noindent\textit{Proof.} Call a subset $F\seq G$ \textit{separated} if $\mu(gA\cap hA)\leq \alpha$ for all distinct $g,h\in F$. Using a straightforward estimate based on inclusion-exclusion, one may show that any separated subset of $G$ is finite. (Details are spelled out in the proof of Claim 1 in \cite[Theorem 5.1]{CP-AVSAR}.)

Now fix $U\in\cU$. Choose $V\in\cU$ such that  $\overline{V\inv V}\seq U$. Let $F$ be a maximal separated subset of $\tau\inv(V)$. As noted above, $F$ is finite. We will show that $\tau_*\inv(V)$ is contained in $F(X\cap \tau_*\inv(U))$, and hence $X\cap \tau_*\inv(U)$ is generic by Claim \ref{cl:generic}. 

Define $C=\{p\in \tau_*\inv(\overline{V\inv V}):\hat{\varphi}(p)\geq\alpha\}$, which is closed in $E$. Then $C\seq X\cap \tau_*\inv(U)$ by choice of $\epsilon$ and $V$. So it suffices to show $\tau_*\inv(V)\seq FC$. Suppose this fails. Then $\tau_*\inv(V)\backslash FC$ is a nonempty open set in $E$ and hence contains $\breve{a}$ for some $a\in G$. Since $\tau(a)=\tau_*(\breve{a})\in V$, we have $a\in\tau\inv(V)$. By choice of $F$, there is some $g\in F$ such that $\mu(gA\cap aA)>\alpha$, i.e., $\varphi(g\inv a)>\alpha$. Recall also that $F$ is contained in $\tau\inv(V)$, so $\tau(g)\in V$. Therefore $\tau(g\inv a)=\tau(g)\inv\tau(a)\in V\inv V$. Since $\varphi(g\inv a)=\hat{\varphi}(g\inv\breve{a})$ and $\tau(g\inv a)=\tau_*(g\inv\breve{a})$, it follows that $g\inv\breve{a}\in C$. So $\breve{a}\in gC\seq FC$, which contradicts the choice of $\breve{a}$.\sclqed \medskip

By the subclaim and Fact \ref{fact:generic}, we have that for any $U\in\cU$, $K\cap X\cap \tau_*\inv(U)$ is nonempty. Set $Y=\{p\in E:\hat{\varphi}(p)\geq\epsilon\}$ and let $\cC=\{K\cap Y\cap \tau_*\inv(\overline{U}):U\in\cU\}$. Then $\cC$ is a collection of  closed subsets of $E$ with the finite intersection property, 
so by compactness there is some $p\in K\cap Y$ such that $\tau_*(p)\in \overline{U}$ for all $U\in\cU$. So $\tau_*(p)=u$, i.e.,  $p\circ u=u$. But also $p\circ u=p$ since  $p\in K$ and $u$ is the identity in $K$. So $p=u$. Therefore $u\in Y$, i.e., $\hat{\varphi}(u)\geq\epsilon$. 
\end{proof}

Now, since $\hat{\varphi}|_K\colon K\to [0,1]$ is continuous and $\hat{\varphi}(u)>0$, there is some $U\in\cU$ such that $\hat{\varphi}(x)>0$ for all $x\in\overline{U}$.
Let $B$ be the Bohr$_0$ set $\tau\inv(U)$ in $G$. 

\begin{claim}\label{cl:final}
$\BD(B\backslash AA\inv)=0$.
\end{claim}
\begin{proof}
Let $\nu\in\LM(G)$ be arbitrary. We will show $\nu(B\backslash AA\inv)=0$. The first step is to push $\nu$  from $\beta G$ to a $G$-equivariant regular Borel probability measure on $E$.

Recall that $E$ is a compactification of $G$ via the map $g\mapsto \breve{g}$. So by the universal property of $\beta G$, there is a (unique) continuous function $\rho\colon \beta G\to E$ such that for $g\in G$, if  $[g]\in\beta G$ denotes the principal ultrafilter on $g$, then  $\breve{g}=\rho([g])$. Now let $\nu^*$ be pushforward of $\nu$ along $\rho$, i.e., $\nu^*(X)=\nu(\rho\inv (X))$ for any Borel $X\seq E$. We need to verify $\nu^*$ is $G$-invariant. Since $\nu$ is $G$-equivariant, it suffices to show $\rho$ is $G$-equivariant. Since $G$ acts by homeomorphisms on $\beta G$ and $E$, and the principal ultrafilters are dense in $\beta G$, it suffices to check that $\rho$ is $G$-equivariant on  principal ultrafilters.  For this, note that if $g,a\in G$ then $\rho(g[a])=\rho([ga])=g\breve{a}=g\rho([a])$.

Next consider the Haar measure $\eta$ on $K$. Recall that $K$ is a closed subset of $E$. So $\eta$ determines a regular Borel probability measure $\eta^*$ on $E$ supported on $K$, i.e., $\eta^*(X)=\eta(X\cap K)$ for any Borel $X\seq E$. We claim that $\eta^*$ is $G$-equivariant. 
Since $\eta$ is invariant under the group operation in $K$, it suffices to fix a Borel set $X\seq E$ and show that $gX\cap K=gu\circ (X\cap K)$ for any $g\in G$. But this follows from the fact $K$ is a subgroup and $gp=gu\circ p$ for any $g\in G$ and $p\in K$.

Now, by Theorem \ref{thm:EN}$(c)$, we have $\nu^*=\eta^*$.  Let $D=\{p\in \tau_*\inv(\overline{U}):\hat{\varphi}(p)=0\}$, which is closed in $E$. By choice of $\overline{U}$, and since $\tau_*$ restricts to  the identity on $K$, we have $D\cap K=\emptyset$. So by definition of $\eta^*$, we have $\eta^*(D)=\eta(D\cap K)=0$. Thus $\nu^*(D)=0$ since $\nu^*=\eta^*$. By definition of $\nu^*$, this yields $\nu(\rho\inv(D))=0$.

Finally, consider the clopen set $C=[B\backslash AA\inv]$ in $\beta G$. We want to show $\nu(C)=0$.
By the above, it suffices to show $C\seq \rho\inv(D)$. Suppose this fails. Then $C\backslash \rho\inv(D)$ is a nonempty open set in $\beta G$, hence contains $[g]$ for some $g\in G$. So $g\in B\backslash AA\inv$. Since $g\not\in AA\inv$, we have $A\cap gA=\emptyset$, and so $\hat{\varphi}(\breve{g})=\varphi(g)=\mu(A\cap gA)=0$. Since $g\in B$, we have $\tau_*(\breve{g})=\tau(g)\in U$. So $\breve{g}\in D$, i.e., $[g]\in\rho\inv(D)$, contradicting the choice of $g$. 
\end{proof}

\section{Discussion}\label{sec:remarks}

\subsection{The definition of Banach density}\label{sub:BDsup}

Recall that for an amenable group $G$ and a set $A\seq G$, we defined $\BD(A)=\sup\{\mu(A):\mu\in\LM(G)\}$, which differs from the standard definition using F{\o}lner nets (e.g., \cite[Section 2]{HiStDAS}). The equivalence between these definitions goes back to a result of Chou \cite{ChouTIM} on locally compact $\sigma$-compact amenable topological groups, which was extended to all locally compact amenable topological groups by Paterson \cite[Theorem 4.17]{Patbook}. In \cite{HindStrauss-means}, Hindman and Strauss prove analogous results for discrete amenable semigroups, and explicitly derive the above identity for Banach density (see Theorems 2.12 and 2.14 in \cite{HindStrauss-means}).

On the other hand, note that our proof of Theorem \ref{thm:main} uses only the inequality $\BD(A)\leq\sup\{\mu(A):\mu\in\LM(G)\}$, which is a comparatively much easier exercise involving ultralimits of normalized counting measures obtained from F{\o}lner nets.

\subsection{F{\o}lner and Steinhaus}\label{sub:Steinhaus}

Recall that F{\o}lner's Theorem can be viewed as a discrete analogue of Steinhaus's Theorem \cite{Steinhaus} on positive Lebesgue measure sets in $\R$ (which was generalized to any locally compact group by Weil \cite{WeilST}). Following this analogy, it is natural to ask if Theorem \ref{thm:main} can be proved with the stronger conclusion that $B$ is \emph{entirely} contained in $AA\inv$. This turns out to be false due to a result of K\v{r}\'{\i}\v{z} \cite{Kriz} that leads to  a counterexample in $\Z$. The was then relaxed to ask if $AA\inv$ contains some Bohr set, rather than a Bohr$_0$ set (see \cite[Section 9.2]{BerRuz}). Even for $\Z$, this remained open for some time and was eventually answered in the negative by Griesmer \cite{GriesZ} (who had previously constructed counterexamples in the case that $G$ is an infinite elementary $p$-group \cite{Gries}).

\subsection{F{\o}lner and Jin}\label{sub:Jin}
Beiglb\"{o}ck, Bergelson, and Fish \cite{BBF} prove F{\o}lner's Theorem for countable amenable groups as a step along the way in their generalization of Jin's Theorem, which we now discuss. For motivation, recall that a set $X\seq G$ is \emph{syndetic} if $G$ can be covered by finitely many left translates of $X$. If $G$ is amenable and $A\seq G$ has positive upper Banach density, then it is a standard exercise to show that $AA\inv$ is syndetic. Observe also that if $X\seq G$ is syndetic and ``almost contained" in $Y\seq G$ in the sense that $\BD(X\backslash Y)=0$, then $Y$ is syndetic. Since Bohr sets are syndetic (c.f., Claim \ref{cl:generic}), we can thus interpret Theorem \ref{thm:main} as a strong algebraic explanation for the syndeticity of $AA\inv$ when $\BD(A)>0$.

Jin's Theorem is based on the question of what happens when $AA\inv$ is replaced by $AA$, or even $AB$ (with $A,B$  of positive upper Banach density). Clearly the conclusion of Theorem \ref{thm:main} fails in this case (even after the obvious concession of allowing a Bohr set in place of a Bohr$_0$ set) due to examples like $A=\N$ in $\Z$ where $A+A$ isn't even syndetic. However, Jin \cite{JinAB} proved that if $A,B\seq\Z$ both have positive upper Banach density, then $A+B$ is \emph{piecewise syndetic}, i.e., there is a syndetic set $X\seq \Z$ which is almost contained in $A+B$ in the weaker sense that $X\backslash(A+B)$ has \emph{lower} Banach density $0$. In \cite{BFW}, Bergelson, Furstenberg, and Weiss strengthen this by showing that one can take $X$ to  be a Bohr set, i.e., $A+B$ is \emph{piecewise Bohr}. The generalization of this result to all countable amenable groups is the main result of Beiglb\"{o}ck, Bergelson, and Fish \cite{BBF}. 

The next  question is what happens with Jin's Theorem for an uncountable amenable group. In \cite{DiNLu}, Di Nasso and Lupini use nonstandard analysis to prove a direct generalization of Jin's Theorem in its original form. In particular, they show that if $G$ is any amenable group and $A,B\seq G$ have positive upper Banach density, then $AB$ is piecewise syndetic. However, whether this can be improved to piecewise Bohr in the uncountable case appears to be open. In light of the main result here, is then worthwhile to discuss the strategy in the countable case from \cite{BBF}. From a high level, their proof consists of three steps:
\begin{enumerate}
\item \cite[Corollary 5.3]{BBF} F{\o}lner's Theorem for a countable amenable group.
\item \cite[Lemma 5.4]{BBF} Suppose $G$ is countable amenable, $X\seq G$ is piecewise Bohr, and $X$ is \emph{finitely embeddable} in $Y\seq G$, i.e., every finite subset of $X$ is contained in some right translate of $Y$. Then $Y$ is piecewise Bohr. 
\item \cite[Proposition 4.1]{BBF} Suppose $G$ is countable amenable and $A,B\seq G$ have positive upper Banach density. Then there is some $C\seq G$ of positive upper Banach density such that $CC\inv$ is finitely embeddable in $AB$.  (Thus, since $CC\inv$ is piecewise Bohr by $(1)$,  $AB$ is piecewise Bohr by $(2)$.)
\end{enumerate}
The generalization of $(1)$ to the uncountable case is our main result (though for these purposes one could use the weaker version of Theorem \ref{thm:main} afforded by Hrushovski's Stabilizer Theorem). As for $(2)$, it is not hard to see that the countability assumption is inessential, and the same result holds for uncountable $G$ by redoing the argument from \cite{BBF} with  nets  indexed by the finite subsets of $G$ instead of countable sequences. However, the proof of $(3)$ uses  tools from ergodic theory relying on countability of $G$. The extension of such  tools to the uncountable setting appears to be an active area of research. For example, recent work of Durcik, Greenfeld, Iseli, Jamneshan, and Madrid  \cite{DGIJM} could be one lead for  extending the proofs of (1) and (3) in  \cite{BBF} to uncountable groups. It would be very interesting to develop an  approach to (3) based instead on model theory and more classical topological dynamics.

\subsection{The model-theoretic context}\label{sub:MT}

Consider an amenable group $G$ and a subset $A\seq G$ with $\BD(A)>0$, i.e., $\mu(A)>0$ for some $\mu\in\LM(G)$. 
The proof of Theorem \ref{thm:main} is based on  the  first-order structure in continuous logic obtained by expanding the group $G$ with the function $\varphi(x)=\mu(A\cap xA)$ as a $[0,1]$-valued predicate. Here $G$ is equipped with the discrete metric. The $G$-flow $S$ in the proof then corresponds to the local type space $S_{\theta}(G)$ where $\theta(x,y)\coloneqq \varphi(y\cdot x)$, and weak almost periodicity of $S$ can be accounted for by stability of $\theta(x,y)$ in the sense of continuous logic. For this particular choice of $\theta$, the  role of stability was made explicit by Hrushovski \cite[Proposition 2.25]{HruAG}. However, the general connection between stability and Hilbert spaces is a key part of the development of continuous logic going back to work of Krivine and Maurey \cite{KrMau} (and of course Grothendieck \cite{GroWAP}). 

The Ellis semigroup $E=E(S)$ can also be identified as a local type space, but with respect to the formula $\theta^\sharp(x;y,z):=\theta(x\cdot z,y)$; see \cite[Lemma 3.1]{CP-AVSAR}. The model-theoretic translation of the work of Ellis and Nerurkar \cite{EllNer} is that  $E$ behaves like the type space of a group definable in a stable theory. In particular, $E$ has a semigroup structure and a compact subgroup $K$ consisting of the generic $\theta^\sharp$-types, with principal generic $u$. Moreover,  uniqueness of the $G$-invariant regular Borel probability measure on $E$ provides the analogue of ``stable compact domination" by the generic types in the sense that the measure of any $\theta^\sharp$-formula is the Haar measure of its restriction to $K$ (see \cite[Proposition 3.6]{CP-AVSAR}). If we had started with a sufficiently saturated structure $G$, then one could see $K$ as $G/G^{00}_{\theta^\sharp}$ where $G^{00}_{\theta^\sharp}$ is the local connected component constructed in \cite[Section 5]{CP-AVSAR}. Thus in this case  the compactification $\tau\colon G\to K$ (mapping $g$ to $gu$) in the proof would  be the canonical quotient map. In any case, $\tau$ is ``$\theta^\sharp$-definable" in the model-theoretic sense, and the semigroup homomorphism $\tau_*$ from the proof is the unique extension of $\tau$ to a continuous map on $E$ (viewed as $S_{\theta^\sharp}(G)$).

This brings us to Claim \ref{cl:upos} which shows $\hat{\varphi}(u)>0$ where $\hat{\varphi}\colon E\to [0,1]$ maps $p$ to $p(\varphi)(1)$. Under the model-theoretic translation, the map $p\mapsto p(\varphi)$ from $E$ to $S$ represents restriction from $S_{\theta^\sharp}(G)$ to $S_\theta(G)$, and thus $\hat{\varphi}$  is nothing more than the continuous function on $E$ given by viewing $\varphi$ as a $\theta^\sharp$-formula. The idea behind the proof of Claim \ref{cl:upos}  is clearer when $G$ is saturated, in which case the proof is essentially showing that for $\epsilon>0$ small enough, $G^{00}_{\theta^\sharp} \wedge (\varphi(x)\geq\epsilon)$  is a partial generic $\theta^\sharp$-type,  hence must contain a (complete) generic $\theta^\sharp$-type, which then must be $u$ since $G^{00}_{\theta^\sharp}$ determines a unique generic. In any case, with Claim \ref{cl:upos} in hand, we can then fix an open identity neighborhood $U$ in $K$ such that $\hat{\varphi}$ is positive on the closure $\overline{U}$. We then define $D\seq E$ to be the closed set of $\theta^\sharp$-types in $\tau_*\inv(\overline{U})$ that are zeroes of $\hat{\varphi}$. Since $D$ contains no generic $\theta^\sharp$-types, it has measure $0$ by stable  compact domination of $K$. Finally, we pull  back from $E$ to the Stone-\v{C}ech compactification $\beta G$. By  uniqueness of the measure on $E$, the preimage of $D$ in $\beta G$ has measure $0$ with respect to any $G$-invariant regular Borel probability measure on $\beta G$. Moreover, this preimage contains the clopen set determined by $B\backslash AA\inv$ where $B$ is the Bohr$_0$ set $\tau\inv(U)$. Therefore $\BD(B\backslash AA\inv)=0$.

\subsection{Sets of popular differences}\label{sec:popdiff}
In forthcoming work (communicated to the author shortly after announcing this paper), Griesmer \cite{Gries-draft} uses the theory of positive definition functions to prove a generalization of F{\o}lner's Theorem which applies to certain  ``level sets" in arbitrary discrete groups obtained from unitary representations. When specialized to amenable groups and sets of positive upper Banach density, this recovers Theorem \ref{thm:main}, but with  the full difference set $AA\inv$ replaced by a ``set of popular differences" for $A$. In particular, let $G$ be an amenable group and fix $A\seq G$ and $\mu\in\LM(G)$ with $\mu(A)>0$. Given $\epsilon>0$, we define the popular difference set $D^\mu_\epsilon(A)=\{x\in G:\mu(A\cap xA)>\epsilon\}$. Recall that $AA\inv$ consists of all $x\in G$ for which $A\cap xA$ is nonempty, and thus $D^\mu_\epsilon(A)\seq AA\inv$ for any $\epsilon>0$. Therefore a Bohr$_0$ set almost contained in $D^\mu_\epsilon(A)$ will also be almost contained in $AA\inv$ (in the sense of Theorem \ref{thm:main}). For  details on the significance of popular difference sets in additive combinatorics, see \cite{GriesPDS,Sanders,Wolf}.

In light of this forthcoming work, we point out that after some cosmetic changes, our proof of Theorem \ref{thm:main} also applies to sets of popular differences. More precisely, in the context of Section \ref{sec:proof}, one can fix any $\epsilon<\frac{1}{2}\mu(A)^2$ and find a Bohr$_0$ set $B$ such that $\BD(B\backslash D^\mu_\epsilon(A))=0$. Indeed, note that  $\hat{\varphi}(u)>\epsilon$ by the proof of Claim \ref{cl:upos}, and thus we can instead choose $U$ such that $\hat{\varphi}(x)>\epsilon$ for all $x\in\overline{U}$. Then, in the proof of Claim \ref{cl:final}, redefine $D=\{p\in\tau_*\inv(\overline{U}):\hat{\varphi}(p)\geq\epsilon\}$, and the same argument shows $\nu(\rho\inv(D))=0$. Now redefine $C=[B\backslash D^\mu_\epsilon(A)]$, and similar steps yield $C\seq \rho\inv(D)$ (using the fact that if $g\not\in D^\mu_\epsilon(A)$ then $\hat{\varphi}(\breve{g})=\varphi(g)=\mu(A\cap gA)\geq\epsilon$).

In fact, one can strengthen the proof of Claim \ref{cl:upos} to show $\hat{\varphi}(u)\geq\mu(A)^2$, which allows the above conclusion to work for any $\epsilon<\mu(A)^2$ (this then matches the form of the result communicated by Griesmer \cite{Gries-draft} in the amenable case). To accomplish this, we modify the first paragraph of the proof of the subclaim as follows. Fix $\alpha<\mu(A)^2$ and call $F\seq G$ \emph{separated} if $\mu(gA\cap hA)\leq\alpha$ for all distinct $g,h\in F$. Then we claim that any finite  separated subset of $G$ has size at most $(\mu(A)-\alpha)/(\mu(A)^2-\alpha)$, and hence any separated subset of $G$ is finite. The proof is a standard Cauchy-Schwarz argument. Fix a finite separated set $F=\{g_1,\ldots,g_n\}$. Define the function $f=\sum_{i}\bone_{g_iA}$. Then, working with the linear functional induced by $\mu$, we have
\begin{align*}
\int f\,d\mu &= \sum_i\mu(g_iA)=n\mu(A),\text{ and }\\
\int f^2\,d\mu &=\sum_i\mu(g_iA)+\sum_{i\neq j}\mu(g_iA\cap g_jA)\leq n\mu(A)+n(n-1)\alpha.
\end{align*}
By Cauchy-Schwarz, $\left(\int f\,d\mu\right)^2\leq\int f^2\,d\mu$. Therefore $(n\mu(A))^2\leq n\mu(A)+n(n-1)\alpha$, which simplifies to $n\leq (\mu(A)-\alpha)/(\mu(A)^2-\alpha)$. With this modification, the proof of Claim \ref{cl:upos} is now valid for any arbitrary $\alpha<\mu(A)^2$ (rather than $\alpha=\frac{1}{2}\mu(A)^2$).

\subsection*{Acknowledgements} Thanks to John Griesmer for  several  comments on the literature surrounding F{\o}lner's Theorem, and for allowing me to mention his forthcoming work in \cite{Gries-draft}. Thanks also to the referee for helpful revisions.

\end{document}